\documentclass[12pt]{amsart}
\usepackage{amssymb,epsfig,amsmath,latexsym,amsthm}
\usepackage{graphicx}
\usepackage{enumitem}
\usepackage{comment}
\theoremstyle{plain}
\newtheorem{theorem}{Theorem}[section]
\newtheorem{lemma}[theorem]{Lemma}

\newtheorem{proposition}[theorem]{Proposition}
\newtheorem{corollary}[theorem]{Corollary}

\newtheorem{conjecture}[theorem]{Conjecture}

\theoremstyle{definition}

\newtheorem{remark}[theorem]{Remark}

\newtheorem{example}[theorem]{Example}

\usepackage{epsfig,color}

\makeatother

\theoremstyle{definition}

\def\fnum{equation}

\numberwithin{equation}{section}

\usepackage[hidelinks]{hyperref}
\begin{document}
\title[Index One Minimal Surfaces in Spherical Space Forms]{ Index one minimal Surfaces in Spherical Space Forms}

\author{Celso Viana}
\address{University College London UCL\\Union Building, 25 Gordon Street, London WC1E 6BT}
\email{celso.viana.14@ucl.ac.uk}

\begin{abstract}
We prove that orientable index one minimal surfaces  in spherical space forms with large fundamental group have genus at most two. This confirms a conjecture of  R. Schoen  for an infinite class of $3$-manifolds.
\end{abstract}
\maketitle
\section{Introduction}

The Morse index is an important analytic quantity in the study of  minimal surfaces. Roughly speaking, it counts  the maximal number of directions a minimal  surface can be deformed in order to decrease  its area.  Under this  analytical point of view, the simplest minimal surfaces are those with small index, namely zero or one. Index zero minimal surfaces, also known as stable, are an well studied topic in Differential Geometry. Among classical results we mention the   Bernstein problem on the classification of complete minimal graphs in $\mathbb{R}^n$ and those connecting stable minimal surfaces
and the topology of manifolds admitting  positive scalar curvature metrics due to Schoen-Yau. The existence of     stable minimal surfaces depends on the geometry and topology of the ambient space and is in general  obtained via a minimization procedure. Such surfaces do not exist in manifolds with  positive Ricci curvature. 
Index one minimal surfaces, on the other hand, do exist is this setting  and are produced by the  one parameter min-max construction of Almgren-Pitts and Simon-Smith \cite{CL,KMN,MN,Pitts}. An important feature about these surfaces  is that they provide  optimal geometric Heegaard  splitting of closed $3$-manifolds.

A guiding principle in the theory asserts that in positively curved manifolds, the  index of a minimal hypersurface controls its topology and geometry. For instance, it is proved in \cite{CKM}  that  the set of  minimal surfaces  with bounded index in a closed $3$-manifold with positive scalar curvature cannot
contain sequences of surfaces with unbounded genus or area. More generally, it is conjectured  in \cite{M,N} that if $\Sigma$ is a minimal hypersurface in a closed manifold with positive Ricci curvature $M$, then $\text{index}(\Sigma)\geq C\,b_1(\Sigma)$, where $b_1(\Sigma)$ is the first Betti number of $\Sigma$
and $C$ is a constant which depends only on $M$.
Estimates of this type have been studied by many authors,  see  \cite{ACS,CM,Li,Ros}  and references therein for further discussion. 
These estimates  are, however, far from being optimal when the index is small in general. A related problem is to  describe the geometry and topology of the minimal surfaces with the smallest index.
In this direction, we mention the classical result that  flat planes and the catenoid are, respectively,  the only embedded minimal surfaces   with index zero and one in $\mathbb{R}^3$, see \cite{dCP,FCS,P,Ros,LR}. Similar classification has also been proved in other non-compact flat space forms, see \cite{Ritore}. In higher dimensions, we mention the works \cite{dCRR,Si} on the classification of  compact minimal hypersurfaces with index one in $\mathbb{RP}^{n}$ and $\mathbb{S}^{n}$, respectively.

Using test functions coming from  meromorphic maps and harmonic forms, 
Ros \cite{Ros} proved that  two sided index one minimal surfaces in 3-manifolds with non-negative Ricci curvature have genus $\leq$ 3.   This result is sharp as the P Schwarz's minimal  surface in $\mathbb{R}^3$ projects to a closed minimal surface  with genus three  and index one in  the  cubic $3$-torus \cite{Ross}. On the other hand, when the ambient space has positive Ricci curvature,  the right estimate is given by the following conjecture:

\begin{conjecture}[Schoen \cite{N}]\label{conjecture}
\textit{Let $M^3$ be a closed three manifold with positive Ricci curvature. If $\Sigma$ is an orientable  minimal surface with index one in
$M^3$, then $\text{genus}(\Sigma)\leq 2$.}
\end{conjecture}
 
The interest in this conjecture is in part motivated by its implications for the classification of closed $3$-manifolds. Namely, it is proved in \cite{KMN} that every closed $3$-manifold with positive Ricci curvature contains an  index one minimal surface realizing its Heegaard genus. If Conjecture \ref{conjecture} is true, then this Heegaard genus is at most two. Combining this result with the classification of genus two $3$-manifolds, one  recovers  the following classical result of  Hamilton: 
\begin{theorem}[Hamilton \cite{H}]
\textit{If $(M^3,g)$ is a three manifold with positive Ricci curvature, then  $M\cong \mathbb{S}^3/G$, where $G$ is a finite group of isometries acting freely on $(\mathbb{S}^3,g_0)$.}
\end{theorem}

 \begin{remark}
	With the exception of lens spaces, which has Heegaard genus one, any other spherical space form has Heegaard genus two \cite{Mo}.
\end{remark}

 The list of $3$-manifolds where the Conjecture \ref{conjecture}  is verified is  small.
In  the case of spherical space forms, the only  examples are the sphere $\mathbb{S}^3$, the projective space $\mathbb{RP}^3$, and the lens spaces $L(3,1)$ and $L(3,2)$ \cite{RR1,V}. The conjecture has also  been proved on  sufficiently pinched convex hypersurfaces in $\mathbb{R}^4$, see \cite[Section 5]{ACS}. 

 Our main result  confirms Schoen's Conjecture  in the class of spherical space forms with large fundamental group. 
\begin{theorem}\label{main theorem}
\textit{There exists an integer $p_0$ so that if $\Sigma$ is an orientable  index one minimal surface   embedded in a spherical space form $M^3$ with $|\pi_1(M^3)|\geq p_0$, then $\text{genus}(\Sigma)\leq 2$.}
\end{theorem}
\begin{remark}
The orientability assumption seems to be necessary in Theorem \ref{main theorem}.
It is pointed out in \cite{Ros,Ross1} that for every integer $n$, there are  lens spaces containing nonorientable area minimizing  surfaces with genus greater than $n$.
\end{remark}

Let $M$ be a spherical space form and $\mathcal{O}_M$ the set of closed orientable minimal surfaces embedded $M$. Define $\mathcal{A}_M=\{|\Sigma|:\Sigma \in \mathcal{O}_M\}$. By standard compactness theorems \cite{CS,Sharp}, we have that $\mathcal{A}_M=|\Sigma|$, where $\Sigma\in \mathcal{O}_M$. Moreover, the work of Mazet-Rosenberg \cite{MaR} and Ketover-Marques-Neves \cite{KMN} imply that $\Sigma$ has index one. From the proof of the Willmore conjecture \cite{M} we know  that any orientable minimal surface in the lens space $L(p,q)$ has area bigger than the  Clifford torus. In general, the least area minimal surface might have genus bigger than the Heegaard genus of the $3$-manifold. In the Berger spheres with small  Hopf fibers, the least area minimal surface  has genus one and is congruent to the Clifford torus.  For the class of  spherical space forms  in Theorem \ref{main theorem}, we have:

\begin{corollary}\label{least area surface}
Let $M$ be a spherical space form with large non-abelian fundamental group. If $\Sigma \in \mathcal{O}_M$ satisfies $\mathcal{A}_M=|\Sigma|$, then $g(\Sigma)=2$.
\end{corollary}
The proof of Theorem \ref{main theorem} is inspired by Ritor\'{e} and Ros' work on the compactness of the space of index one minimal surfaces  in flat three torus \cite{RR}. Among other results, they  proved that  the flat  three torus with small injective radius and unit volume do not contain orientable index one minimal surfaces. Following similar ideas, we show that any rescaled sequence of index one minimal surfaces with genus three in spherical space forms with large fundamental group converges to a totally geodesic surface in a flat $3$-manifold. We contradict this statement by showing that the curvature of such surfaces is large somewhere by  an application of the  Rolling Theorem.

\noindent\textit{Acknowledgements.}
I am grateful to my  advisor  Andr\'{e} Neves for his constant encouragement and support. 
This work was carried out while  visiting the University of Chicago and I thank the Mathematics department  for its hospitality.
I also thank Lucas Ambrozio and Ezequiel Barbosa for helpful conversations. 

This work was supported by the Engineering and
Physical Sciences Research Council [EP/L015234/1], and the EPSRC Centre for Doctoral
Training in Geometry and Number Theory (London School of Geometry and
Number Theory), University College London.
\section{Preliminaries}
\subsection{Morse index}  A  surface $\Sigma \subset (M^3,g)$
is called minimal when the trace of its second fundamental
form is identically zero. Equivalently, the first variation of its area
is zero for all variations  generated by flows of compact supported vector fields $X\in \mathcal{X}_0(M)$. If $\Sigma$ is two sided, then its  second variation formula  is given by:
\begin{equation}
I(f):= \frac{d^2}{dt^2}\bigg|_{t=0}\text{area}(\Sigma_t)= \int_{\Sigma} |\nabla f|^2 - (\text{Ric}(N,N)+ |A|^2)f^2\, d_{\Sigma},
\end{equation}
where $f=\langle X,N\rangle$ is the normal component of $X$, $\text{Ric}(\cdot,\cdot)$ is the Ricci curvature of $M$ , and $A$ is the second fundamental form of $\Sigma$.
The quantity $I(f)$ is called the  Morse index form of $\Sigma$ and is the quadratic form associated to the Jacobi operator \[
L= \Delta + (\text{Ric}(N,N)+ |A|^2).
\]
 The Morse index of $\Sigma$ is defined as the number of negative eigenvalues of $L$. 

\subsection{Spherical space forms}
We regard $\mathbb{S}^3$ as the unit quaternions, i.e., $(z,w)=z_1 + z_2\,i + (w_1+w_2\,i)\,j$ and $|z|^2+|w|^2=1$. 
Let $\phi: \mathbb{S}^3\times \mathbb{S}^3\rightarrow SO(4)$ be the homomorphism of groups which associate for each pair $(u_1,u_2)\in  \mathbb{S}^3\times \mathbb{S}^3$ the isometry $\phi(u_1,u_2)\in SO(4)$ given by $x\mapsto\phi(u_1,u_2)(x)=u_1xu_2^{-1}$. The map $\phi$ is  surjective and $\text{Ker}\phi=C=\{(\pm1,\pm 1)\}$ .
Similarly, one can construct the homomorphism  $\psi: \mathbb{S}^3\rightarrow SO(3)\subset SO(4)$ defined by $x\in\mathbb{S}^3 \mapsto \psi(u)(x)=uxu^{-1}$. This map  is also surjective and its kernel is $\{\pm 1\}$. It follows that there exists an unique homomorphism $\varphi: SO(4)\rightarrow SO(3)\times SO(3)$ such that $\varphi\circ \phi= \psi\times\psi$.

For each finite subgroup $G\subset SO(4)$ we associate $H=\varphi(G)\subset SO(3)\times SO(3)$. The projection of $H$ on each factor of $SO(3)\times SO(3)$ is denoted by $H_1$ and $H_2$  respectively. If  $G$ acts freely on $\mathbb{S}^3$, then  $H_1$ or $H_2$ must be cyclic \cite{Sc}. The pre-images in $\mathbb{S}^3$ of $H_1$ and $H_2$ via the homomorphism $\psi$ are denoted by $\widehat{H}_1$ and $\widehat{H}_2$ respecively.
Since $H_1$ and $H_2$ are  finite subgroups of $SO(3)$, they must be isomorphic to   either the cyclic group, the dihedral group  $D_n$, the tetrahedral group $T$, the octahedral group $O$ or  the icosahedral group $I$. 

It is showed in \cite{S} that any finite subgroup $G\subset SO(4)$  is conjugated in $SO(4)$ to a finite subgroup of either $\phi(\mathbb{S}^1\times \mathbb{S}^3)$ or  $\phi(\mathbb{S}^3\times \mathbb{S}^1)$. Two  important remarks that we will use are the following:  $\phi(\mathbb{S}^1\times\mathbb{S}^3)$ preserves the Hopf fibers in $\mathbb{S}^3$ and  left multiplication by unit quaternions leaves the Hopf fibers invariant. Recall that the Hopf map $h:\mathbb{S}^3 \rightarrow \mathbb{S}^2(\frac{1}{2})$ sends
$(z,w)$ to $z/w$ where we think of $\mathbb{S}^2$ as $\mathbb{C}\cup\{\infty\}$. In particular, up to conjugation in $O(4)$, we may assume that  $G$ is a subgroup of $\phi(\mathbb{S}^1\times\mathbb{S}^3)$.
The following describes the classical classification of $3$-dimensional spherical space forms:

\begin{theorem}\label{classification of elliptic manifolds}
	\textit{Let $G$ be a finite subgroup of $\phi(\mathbb{S}^1\times \mathbb{S}^3)$ acting freely on $\mathbb{S}^3$. Then one of the following holds:
		\begin{enumerate}
			\item $G$ is cyclic;
			\item $H_2$ is $T$, $O$, $I$, or $D_n$ and $H_1$ is cyclic of order coprime to the order of $H_2$. Moreover, $G=\phi(\widehat{H}_1\times \widehat{H}_2)$;
			\item $H_2=T$ and $H_1$ is cyclic. Moreover, $G$ is a subgroup of index three in $\phi(\widehat{H}_1\times \widehat{H}_2)$;
			\item $H_2=D_n$ and $H_1$ is cyclic. Moreover, $G$ is a subgroup of index two in $\phi(\widehat{H}_1\times \widehat{H}_2)$.
	\end{enumerate}}
\end{theorem}  
\begin{proof}
	See page 455 in \cite{S}.
\end{proof}

The spherical space forms obtained when $G$ is cyclic are called  lens spaces. If $p$ and $q$ are relative primes, then we denote by $L(p,q)$ the lens space  defined  by the action of  $\mathbb{Z}_p$  on $\mathbb{S}^3$ as follows: given $m\in \mathbb{Z}_p$, we define  $m\cdot(z,w)= (e^{2\pi \frac{mq\,i}{p}}z,e^{2\pi \frac{m\,i}{p}}w)$. 
The  Clifford torus $T_r\subset \mathbb{S}^3$, defined as
\begin{equation}\label{Clifford torus}
	T_r:=\mathbb{S}^1(\cos(r))\times\mathbb{S}^1(\sin(r))\subset \mathbb{S}^3
\end{equation}
where $r\in [0,\frac{\pi}{2}]$, is invariant by the group $\mathbb{Z}_p$ and the  projection of this family foliates $L(p,q)$ by tori of constant mean curvature. One can check that $T_{\frac{\pi}{4}}$  projects to an index one  minimal tori in $L(p,q)$ for every $p\geq 2$ and $q\geq 1$.

\begin{example}[Immersed index one minimal tori]
	Let $T$ be a Clifford torus in $\mathbb{S}^3$  containing the geodesics $T_0$ and $T_{\frac{\pi}{2}}$. For each $p$, let  $\mathcal{V}_p$ be the varifold defined by $\mathcal{V}_p= \cup_{ g\in \mathbb{Z}_p} g\cdot T$, where $\mathbb{Z}_p$ is the group defined above. The projection of $\mathcal{V}_p$ in $L(p,q)$ is a  minimal immersed  torus which fails to be embedded at the critical fibers $T_0$ and $T_{\frac{\pi}{2}}$ when $q\neq 1$. If $p$ is even, then  $\text{Index}(\mathcal{V}_p/\mathbb{Z}_p)=1$. Moreover, if $p,q$ are chosen so that $\lim_{p \rightarrow \infty} \text{diam}(T_{\frac{\pi}{4}}/ \mathbb{Z}_p)=0$, then the varifold $\mathcal{V}=\lim_{p\rightarrow \infty}\mathcal{V}_p$  is the foliation of $\mathbb{S}^3$ where the leaves are Clifford torus containing $T_0$ and $T_{\frac{\pi}{2}}$.
\end{example}

\begin{remark}[Doubling the Clifford torus]
If the group $G$ satisfies   item (4) in Theorem \ref{classification of elliptic manifolds}, i.e., $H_2=D_n$, then we call $\mathbb{S}^3/G$ a Prism manifold. These spherical space forms are double covered by  lens spaces.  In particular, one can  sweep-out  each Prism manifold with surfaces whose area does not exceed twice the volume of $\mathbb{S}^3/G$, see \cite{KMN}. Applying the min-max theory, one obtains an orientable index one minimal surface  with area bounded as above. If the order of $G$ is sufficiently large, then  Theorem \ref{main theorem} implies that the genus of these min-max surfaces is two. We remark that these manifolds do not contain minimal tori by Frankel's Theorem \cite{F}. One can visualize these surfaces better when the Prism manifold they live have a double cover $L(p,q)$ which satisfies  $\lim_{p \rightarrow \infty} \text{diam}(T_{\frac{\pi}{4}}/ \mathbb{Z}_p)=0$.  In this case, the orbit of a point $x\in T_r$ with respect to $G_p$ is becoming dense in the Clifford torus $T_r$. For every $p$, let $\hat{\Sigma}_p$ be the pre-image of these index one minimal surfaces in $\mathbb{S}^3$. By Frankel's Theorem $\hat{\Sigma}_p\cap T_{\frac{\pi}{4}}\neq \emptyset$ for every $p$; hence, $\hat{\Sigma}_p$ converges as varifolds  to the  Clifford torus $T_{\frac{\pi}{4}}$ with multiplicity two. The surface $\hat{\Sigma}_p$ pictures like a  doubling of the minimal Clifford torus.
\end{remark}

\begin{remark}[Desingularizing stationary varifolds]
	Another family of spherical space forms  is given by the quotients
	$\mathbb{S}^3/(I^{*}\times \mathbb{Z}_m)$, where $m$ satisfies $(m,30)=1$. By Frankel's Theorem, there are no minimal spheres or minimal tori in $\mathbb{S}^3/(I^{*}\times \mathbb{Z}_m)$. With the help of the Hopf fibration $h:\mathbb{S}^3\rightarrow \mathbb{S}^2$, it is possible to construct a sweep-out  of $\mathbb{S}^3$  which is invariant by $I^{*}\times \mathbb{Z}_m$ and that  projects to a sweep-out in $\mathbb{S}^3/(I^{*}\times \mathbb{Z}_m)$ by surfaces with genus two and area bounded from above by $\frac{C}{m}$, see \cite[Section 6]{K}. Applying the min-max theory, one obtains an index one minimal surface  $\Sigma_m$ with genus two in $\mathbb{S}^3/(I^{*}\times \mathbb{Z}_m)$ and    area satisfying $|\Sigma_m|\leq \frac{C}{m}$. Its pre-image   $\hat{\Sigma}_m\subset \mathbb{S}^3$  has uniform bounded area and converges, as $m\rightarrow\infty$, to a stationary varifold $\mathcal{V}$ which is invariant by the Hopf fibration. In particular, $\mathcal{V}=h^{-1}(\mathcal{T})$, where   $\mathcal{T}$ is a $I^{*}$ invariant geodesic net in $\mathbb{S}^2$. By Allard's Regularity Theorem,
	 the genus of $\hat{\Sigma}_m$ is concentrated near $h^{-1}(V)$, where $V$ is the set of vertices of $\mathcal{T}$. The surface $\hat{\Sigma}_m$ pictures like a desingularization of $h^{-1}(\mathcal{T})$ near $h^{-1}(V)$ through Scherk towers.
\end{remark}

\subsection{Non compact flat space forms}
Every non-compact orientable flat space form  is the quotient of $\mathbb{R}^3$ by a discrete subgroup $G$
of the group $\text{Iso}(\mathbb{R}^3)$ of affine orientation preserving isometries acting properly and  discontinuously   in $\mathbb{R}^3$.  For every subgroup $G$ we denote by $\Gamma(G)$ the subgroup of translations in $G$.  The following describe all the possible types of affine diffeomorphic complete non compact orientable flat three manifolds (see \cite{W} for a comprehensive discussion):

If $\text{rank}(\Gamma(G)) = 0$ or $1$, then either $G =\{\text{Id}\}$ or $G= S_{\theta}$, with $0\leq\theta\leq \pi$, where $S_{\theta}$ is the subgroup generated by a screw motion given by a rotation of angle $\theta$
followed by a non trivial translation in the direction of the rotation axis.

If $\text{rank}(\Gamma(G)) =2$, then  either $G$ is
generated by two linearly independent translations and $\mathbb{R}^3/G$ is the
Riemannian product $T^2\times\mathbb{R}$, where $T^2$ is a flat torus, or $G$ is
generated by a screw motion with angle $\pi$ and a translation orthogonal to the axis
of the screw motion.

\begin{theorem}[Ritor\'{e} \cite{Ritore}]\label{curvature of index one minimal surface}
	\textit{If  $\Sigma$ is a complete orientable index one minimal surface  properly embedded in a non-compact orientable flat  $3$-manifold $\mathbb{R}^3/G$ , then \[-8\pi<\int_{\Sigma}K_{\Sigma}\, d_{\Sigma}\leq-2\pi.\]}  
\end{theorem}

\begin{remark}\label{multiple 2pi}
	By \cite{MR,MR1}, the total curvature of a properly embedded minimal surface in $\mathbb{R}^3/G$ 
	is a multiple of $2\pi$ if finite.
\end{remark}
\begin{example}[Index one Helicoids  with total curvature $-2\pi$] Let $\Sigma$ the helicoid in $\mathbb{R}^3$ parametrized by  $X(u,v)=(u\cos(v),u\sin(v), v)$. One can check that 
	\begin{eqnarray}\label{curvature helicoid}
	\int_{\Sigma\cap\{0\leq v\leq 4\pi\}}K_{\Sigma} d_{\Sigma}= -4\pi.
	\end{eqnarray}
	Now consider $\Sigma/\mathbb{Z}_{4\pi}$   in $\mathbb{R}^3/\mathbb{Z}_{4\pi}$, where $\mathbb{Z}_{4\pi}$ is the group of vertical translations by multiples of $4\pi$. Recall that the Gauss map $N:\Sigma/\mathbb{Z}_{4\pi}\rightarrow \mathbb{S}^2$ is conformal and with degree one by (\ref{curvature helicoid}). A standard argument implies that $\text{ind}(L_{\Sigma/\mathbb{Z}_{4\pi}})=\text{ind}(L_0)$, where $L_{\Sigma}= \Delta + |\nabla N|^2$ is the Jacobi operator of $\Sigma$ and $L_0$ is the operator $L_0= \Delta+ 2$ on  $\mathbb{S}^2$. Hence, $\text{ind}(\Sigma/\mathbb{Z}_{4\pi})=1$. Let $S_{\pi}$ be the subgroup of isometries generated by the screw motion $R(x,y,z)= (-x,-y, z+ 2\pi)$.  Using that $S_{\pi}$ is a subgroup of order two in $\Sigma/\mathbb{Z}_{4\pi}$, we conclude that  $\Sigma/S_{\pi}$ is a minimal surface with index one and total curvature $-2\pi$ in $\mathbb{R}^3/S_{\pi}$.
\end{example}
\begin{remark}
	It is an open question weather there exists an index one minimal surface $\Sigma$ in $\mathbb{R}^3/G$ such that $\int_{\Sigma}K_{\Sigma}\,d_{\Sigma}=-6\pi$. 
	If $\Sigma$ is an index one minimal surface in a non-compact flat $3$-manifold  $\mathbb{R}^3/G$ where $G$ contains only translations, then $\int_{\Sigma}K_{\Sigma}\,d_{\Sigma}=-4\pi$ \cite{RR}.
\end{remark}
\begin{example}
	Let us show that  $\mathbb{R}^3/S_{\frac{2\pi}{l}}$ can be obtained as a limit of Lens spaces under  the Cheeger-Gromov convergence. To see this, consider the sequence of Lens spaces $(L(p_k,k), p_k^2\,g_0,x_k)$, where   $x_k$ lies on the critical fiber $T_{\frac{\pi}{2}}$ and $p_k=l(k-1)$. This sequence has curvature close to zero and injective radius at $x_k$ bounded from below by $\pi$. We claim that
	\[
	(L(p_k,k), p_k^2\,g_0,x_k)\xrightarrow{C-G}(\mathbb{R}^3/S_{\frac{2\pi}{l}}, \delta, x_{\infty}).
	\]
	The observation is that the critical fiber $T_{\frac{\pi}{2}}$ has length $2\pi$ whereas the nearby Hopf fibers are equidistant  and  have length $2\pi\,l$.
\end{example}

\section{Proof of Theorem \ref{main theorem}}
\begin{proposition}\label{Rolling theorem}
\textit{Let $\Sigma$ be a  closed minimal surface in $\mathbb{S}^3$ and $N: \Sigma_p \rightarrow \mathbb{S}^3$ be the unit normal vector field of $\Sigma$. If we denote by
\begin{eqnarray*}c=c(\Sigma)=\min\bigg\{ \arctan\bigg(\frac{1}{\max\{\lambda_2(x): x\in \Sigma\}}\bigg),\frac{\pi}{4}\bigg\},
	\end{eqnarray*}
where $\lambda_2(x)$ is the non-negative principal curvature at $x$, and  by $F:\Sigma\times [0,c)\rightarrow \mathbb{S}^3$ the exponential map on $\Sigma$, which is given by
\[(x,t)\mapsto F(x,t)=\cos(t)x+ \sin(t)N(x),\]
then $F$ is a diffeomorphism onto its image.}
\end{proposition}
\begin{proof}
We may assume that $g(\Sigma)\geq 1$ since a minimal sphere in $\mathbb{S}^3$ is an equator and the Proposition trivially holds.

Let $\{e_1,e_2\}$ be an orthonormal basis with  eigenvectors of the second fundamental form $A_{\Sigma}$ and $\{\lambda_1,\lambda_2\}$ the respective eigenvalues. It follows that $dF(e_i)=(\cos(t)-\sin(t)\lambda_i)e_i$ and $dF(\partial t)=-\sin(t)x+ \cos(t)N(x)$. Since $\tan(t)\leq 1/\max_{\Sigma}\{\lambda_2\}$ for every  $t\in(0,c)$, we conclude that $F$ is a local diffeomorphism. The unit normal vector field along $\Sigma_t=F(\Sigma,t)$ is $N_t=-\sin(t)x+ \cos(t)N(x)$. Moreover, if we denote the mean curvature of $\Sigma_t=$ by $H_t$, then 
\[H_t=\frac{1}{2} \frac{(1+\lambda_2^2)\sin(2t)}{(\cos^2(t)-\sin^2(t)\lambda_2^2)}> 0,\]
for every $t\in(0,c)$.
Let $t_0=\sup\{t>0: F:\Sigma\times [0,t]\rightarrow \mathbb{S}^3\, \text{is injective}\}$. If $t_0<c$, then there exist $(x_1,t_1)$ and $(x_2,t_2)$ in $\Sigma\times[0,t_0]$ with the same image under $F$. Since $\Sigma$ separates $\mathbb{S}^3$,  these points must lie on $\Sigma\times\{t_0\}$. Hence, we may assume that $F(x_1,t_0)=F(x_2,t_0)$ and that $x_1\neq x_2$. Since $\Sigma_{t_0}$ has a  tangential self intersection at $F(x_1,t_0)$, we conclude that $N_{t_0}(x_1)=\pm N_{t_0}(x_2)$. If $N_{t_0}(x_1)= N_{t_0}(x_2)$, then $t_0= \frac{\pi}{4}$, contradiction. Consequently, $N_{t_0}(x_1)=-N_{t_0}(x_2)$ since $x_1\neq x_2$. Hence,  $\Sigma_{t_0}$ is locally at $F(x_1,t_0)$, an union of two tangential surfaces $\Gamma_1$ and $\Gamma_2$ with $\Gamma_1\leq \Gamma_2$. Moreover, the mean curvatures in the $N_{t_0}(x_1)$ direction say satisfies $H_{\Gamma_1}\leq 0 \leq H_{\Gamma_1}$. Applying the Maximum Principle  \cite[Lemma 1]{S}, we conclude 
the existence of neighborhoods of $x_1$ and $x_2$ in $\Sigma$ with the same image under $F_{t_0}$ and $H_{t_0}=0$ there. This is a contradiction and the result follows.
\end{proof}

\begin{lemma}\label{local area bounds}
\textit{Let $\Sigma$ be a orientable  minimal surface embedded in $\mathbb{S}^3$. If $R< c(\Sigma)$, then there exists $C>0$ independent of $\Sigma$ such that  $\text{vol}(B_{2R}(x))\geq C\,R\cdot\text{area}(\Sigma\cap B_R(x))$ for every $x\in \Sigma$.}
\end{lemma}

\begin{proof}
By Lemma \ref{Rolling theorem}, the following map is a diffeomorphism onto its image:
\[F:\Sigma\cap B_R(x)\times [0,\frac{R}{2})\rightarrow \mathbb{S}^3.\]
Let us denote the image by $\Omega$. By the change of variables formula,
\[\text{Vol}(\Omega)=\int_{0}^{\frac{R}{2}}\int_{\Sigma\cap B_R(x)}\bigg(\cos^2(s)-\lambda_2^2\sin^2(s)\bigg)d_{\Sigma}\,ds.\]
Hence, we can choose $0<C<\min\{\cos^2(s)(1-\frac{\tan^2(s/2)}{\tan^2(s)}): s\in [0,\frac{c}{2}]\}$ such that
 $\text{Vol}(\Omega)\geq C R \text{Area}(\Sigma\cap B_R(x))$. As $\Omega \subset B_{2R}(x)$, the lemma is proved.
\end{proof}

Let $\{\Sigma_n\}$ be a sequence of minimal hypersurfaces in a Riemannian manifold $(M,g)$. We say that $\{\Sigma_n\}$  converges, in the $C^{\infty}$ topology, to a surface $\Sigma$ if for every $x\in\Sigma$ and for $n$ large, the hypersurface $\Sigma_n$ can be written locally as  graphs over an open set of $T_p\Sigma$, and these graphs converge smoothly to the graph of $\Sigma$.

We say that $\{\Sigma_n\}$  satisfy  \textit{local area bounds} if there exist $r>0$ and $C>0$ such that $|\Sigma_n\cap B_r(x)|\leq C$ for every $x\in M$.

\begin{proposition}\label{compactness}
	\textit{Let $\{\Sigma_n\}\subset(M,g_n) $ be a sequence of properly embedded minimal surfaces  such that $\sup_{\Sigma_n}|A_n|\leq C$ and  with local area bounds. Assume  that $g_n$ converges to $g$, in the $C^{\infty}$ topology.} 
	
	\textit{If $\{\Sigma_n\}_{n=1}$ has an accumulation point,  then we can extract a  subsequence which converges  to a  minimal surface $\Sigma$ properly embedded in $(M,g)$.}
\end{proposition}

\begin{lemma}\label{total curvature inequality}
\textit{Let $M_p$ be a $3$-manifold with positive Ricci curvature and $\Sigma_p\subset M_p$ a closed orientable minimal surface with index one and genus $h$. Assume  that $(M_p,g_{p},x_p)$ converges, in the Cheeger-Gromov sense, to a flat manifold $(M,\delta,x_{\infty})$
	and that $(\Sigma_p,x_p)$ converges graphically with multiplicity one to a properly embedded  minimal surface $(\Sigma_{\infty},x_{\infty})$ in $(M,\delta,x_{\infty})$. 
	\begin{enumerate} 
		\item If $h=2$, then $\int_{\Sigma_{\infty}}K_{\infty}\,d_{\Sigma_{\infty}}=-6\pi$, $-4\pi$, or $0$. 
		\item[] 
		 \item If $h=3$, then $\int_{\Sigma_{\infty}}K_{\infty}\,d_{\Sigma_{\infty}}=0$.
\end{enumerate}}
\end{lemma} 
\begin{proof}
	Since the Cheeger-Gromov convergence preserves  topology  in the compact setting, we conclude that $M$ is non compact.  It follows that $\Sigma_{\infty}$ is a complete non compact minimal surface in $M$ since $h\geq 2$. The  multiplicity one convergence implies that  $\Sigma_{\infty}$ is two sided. Moreover, the index of    $\Sigma_{\infty}$ is at most one by the lower semi continuity of the index. If $\text{Ind}(\Sigma_{\infty})=0$, then $\Sigma_{\infty}$ is flat and we are done.  Hence, we may assume that   $\text{Ind}(\Sigma_{\infty})=1$. 
	By  classical arguments in \cite{C},  $\Sigma_{\infty}$ is conformally equivalent to $\Sigma-\{q_1,\ldots,q_l\}$, where $\Sigma$ is a closed Riemann surface. Let $D_i(q_i)$ be conformal disks on $\Sigma$ centered at $q_i$. Given $\varepsilon>0$ we define $U_{\varepsilon}$ to be 
	$\Sigma-\cup_{i=1}^l\{z\in D_i(q_i);|z_i|\leq \varepsilon\}$.  On the  set  $U_{\varepsilon^2}$ we define the function $u_{\varepsilon}$  by:
	\[u_{\varepsilon}=0\quad\text{on}\quad U_{\varepsilon} \quad \text{and}\quad u_{\varepsilon}=\frac{\ln(\frac{|z|}{\varepsilon})}{\ln(\varepsilon)}\quad \text{for}\quad z \in  U_{\varepsilon^2}-U_{\varepsilon}.\]
	One can check that $\lim_{\varepsilon \rightarrow 0}\int_{\Sigma}|\nabla u_{\varepsilon}|^2 d_{\Sigma}=0$.
	The set $U_{\varepsilon}$ is seen as a subset of $\Sigma_{\infty}$ and,  by choosing  $\varepsilon$   small, we   may assume that $\text{Index}(U_{\varepsilon})=1$. It follows that for $p$ large depending on $\varepsilon$, there exist $U_p\subset U_p^{\prime} \subset \Sigma_p$ for which $\text{Index}(U_p)=1$ and such that $U_p$ and $U_p^{\prime}$ converge graphically to $U_{\varepsilon}$ and $U_{\varepsilon^2}$, respectively.
	Moreover, by means of $u_{\varepsilon}$ we can construct, for  each $p$ large enough,  an function $u_p$ on $\Sigma_p$ satisfying $u_p=0$ on $U_p$, $u_p=1$ at $\Sigma_p- U_p^{\prime}$, and such that $\lim_{p\rightarrow \infty}\int_{\Sigma_p}|\nabla u_p|^2\,d_{\Sigma_p}=0$, i.e., $\int_{\Sigma_p}|\nabla u_p|^2\,d_{\Sigma_p}=O_1(\varepsilon)$. 
	As  $\Sigma_p$ and $U_p$ both have Index one, we concluded that $\text{Indice}(\Sigma_p-U_p)=0$.
	As $\text{supp}(u_p)\subset \Sigma_p-U_p$, we obtain
		\[0\leq \int_{\Sigma_p}\bigg(|\nabla u_p|^2-(\text{Ric}_{g_p}(N_p,N_p)+ |A_p|^2)\,u_p^2\bigg)\,d_{\Sigma_p}.\]
		By the Gauss equation, $2\overline{K}_p= 2K_p+ |A_p|^2$, where $\overline{K}_p$ is the sectional curvature of $M$ in the direction of $T\Sigma_p$. Therefore,
			\begin{eqnarray*}
		0&\leq&
		\int_{\Sigma_p} \bigg(|\nabla u_p|^2 - (\text{Ric}_{g_p}(N_p)+ 2\overline{K}_p)\,u_p^2 + 2K_pu_p^2\bigg)\,d_{\Sigma_p}\\
		&=&\int_{\Sigma_p} (|\nabla u_p|^2d_{\Sigma_p} +2\int_{\{K_p\leq 0\}}K_pu_p^2 )d_{\Sigma_p} +  2\int_{\{K_p>0\}} |K_p|\,u_p^2 d_{\Sigma_p} \\
		&& -\int_{K_p\leq 0}  (\text{Ric}_{g_p}(N_p)+ 2\overline{K}_p)u_p^2 - \int_{K_p> 0}  (\text{Ric}_{g_p}(N_p)+ 2\overline{K}_p)u_p^2.
	\end{eqnarray*}
If $\{e_1,e_2\}$ is an orthonormal base for $T\Sigma_p$, then $\text{Ric}_{g_p}(e_1)= \overline{K}_p + \overline{K}(e_1,N)$, $\text{Ric}_{g_2}(e_2)=\overline{K}_p +\overline{K}_{p}(e_2,N)$, and $\text{Ric}_{g_p}(N_p)=\overline{K}(e_1,N) + \overline{K}(e_2,N)$. This immediately  implies that $2\overline{K}+ \text{Ric}_{g_p}(N_p)= \text{Ric}_{g_p}(e_1)+ \text{Ric}_{g_p}(e_2)$. Hence,
\begin{eqnarray*}
	0&\leq & \int_{\Sigma_p} |\nabla u_p|^2d_{\Sigma_p} +2\int_{\{K_p\leq 0\}}K_pu_p^2 d_{\Sigma_p} \\
	&& \quad\quad\quad\quad\quad\quad +  \int_{\{K_p>0\}} \bigg(2|K_p| -\text{Ric}_{g_p}(N_p)- 2\overline{K}_p \bigg)\,u_p^2 \\
	&=& \int_{\Sigma_p} |\nabla u_p|^2d_{\Sigma_p} +2\int_{\{K_p\leq 0\}}K_pu_p^2 d_{\Sigma_p} \\
&&\quad  \quad \quad \quad\quad \quad \quad \quad \quad-	\int_{\{K_p>0\}}\bigg( |A_p|^2+\text{Ric}_{g_p}(N_p)\bigg)\,u_p^2.\\
		&\leq& \int_{\Sigma_p} |\nabla u_p|^2d_{\Sigma_p} +  \int_{\{u_p\equiv 1\}\cap\{K_p\leq 0\}} 2K_p\,d_{\Sigma_p} \\
		&=& O_1(\varepsilon) +  \int_{\Sigma_p\cap \{K_p\leq 0 \}}2K_p\,d_{\Sigma_p}
		- \int_{U_p\cap \{K_p\leq 0\} }2K_p\, d_{\Sigma_p}.
	\end{eqnarray*}
	On the other hand, we have that $\int_{U_p\cap \{K_p\leq 0\}}K_p\,d_{\Sigma_p}= \int_{\Sigma_{\infty}}K_{\infty}d_{\Sigma_{\infty}}  +O_2(\varepsilon)$, for the total curvature of $\Sigma_{\infty}$ is uniformly close to that of $U_{\varepsilon}$ which is uniformly close to that of $\int_{U_p\cap \{K_p\leq 0\}}K_p\,d_{\Sigma_p}$. Hence,
	\[\int_{\Sigma_{\infty}} K_{\infty}\,d_{\Sigma_{\infty}} \leq O_1(\varepsilon)+O_2(\varepsilon) + \int_{\Sigma_{p}\cap \{K_p\leq 0\}}K_{p}\,d_{\Sigma_{p}}.\]
	This implies that $\int_{\Sigma_{\infty}} K_{\infty}\,d_{\Sigma_{\infty}} \leq  \int_{\Sigma_{p}\cap \{K_p\leq 0\}}K_{p}\,d_{\Sigma_{p}}$ since $\int_{\Sigma_{\infty}} K_{\infty}d_{\Sigma_{\infty}}$ and $\int_{\Sigma_{p}\cap \{K_p\leq 0\}}K_{p}d_{\Sigma_p}$ are independent of $\varepsilon$. On the other hand, \[\int_{\Sigma_{p}\cap \{K_p\leq 0\}}K_{p}\,d_{\Sigma_{p}}\leq \int_{\Sigma_{\infty}} K_{\infty}\,d_{\Sigma_{\infty}}\] by the upper semi continuity of the limit of non-positive functions. Therefore,
	\begin{equation}\label{contradiction}
	-8\pi< \int_{\Sigma_{\infty}} K_{\infty}\, d_{\Sigma_{\infty}}= \lim_{p \rightarrow\infty} \int_{\Sigma_{p}\cap \{K_p\leq 0\}}K_p\,d_{\Sigma_p}\leq 4\pi(1-h).
	\end{equation}
	The first strictly inequality is from Theorem \ref{curvature of index one minimal surface} and the second inequality if from the Gauss-Bonnet Theorem. If $h=2$, then $\int_{\Sigma_{\infty}}K_{\Sigma_{\infty}}\,d_{\Sigma_{\infty}}=-4\pi$ or $-6\pi$ by Remark \ref{multiple 2pi}. If $h=3$, then  (\ref{contradiction}) becomes a contradiction and the lemma is proved.
	\end{proof}

\begin{corollary}
\textit{If $M_p$ is a spherical space form and $\text{genus}(\Sigma_p)=2$, then
\[
\int_{\Sigma_{\infty}}K_{\infty}\,d_{\Sigma_{\infty}}=-4\pi\quad \text{or}\quad 0.
\] }
\end{corollary}
\begin{proof}
Since there is no loss of negative Gaussian curvature, then \[
\lim_{p \rightarrow \infty}\int_{\{K_p>0\}} (2K_p -\text{Ric}_{g_p}(N_p)- 2\overline{K}_p\,d_{\Sigma_p}=0.
\]
By the scale invariance of this quantity, we can assume that $\overline{K}_p=1$. The Gauss equation then implies that $\lim_{p\rightarrow\infty} |\Sigma_p\cap \{K_p>0\}|=0$. Hence, $\lim_{p \rightarrow \infty}\int_{\{K_p>0\}} K_p \,d_{\Sigma_p}=0$. The corollary now follows from the Gauss-Bonnet Theorem.
\end{proof}

\begin{theorem}
\textit{There exists an integer  $p_0$ such that if $\Sigma$ is an orientable index one minimal surface in an spherical space form $M^3$  with $|\pi_1(M)|\geq p_0$, then $\text{genus}(\Sigma)\leq 2$.}
\end{theorem}

\begin{proof}
In what follows $M_p$ denotes an spherical space form such that $|\pi_1(M_p)|=p$, i.e., $M_p=\mathbb{S}^3/G_p$ and $|G_p|=p$.
Arguing by contradiction, let us assume the existence of  a sequence of spherical space forms $\{M_{p_i}\}_{i=1}^{\infty}$ such that  each $M_{p_i}$ contains an index one minimal surface $\Sigma_{p_i}$ of genus three and that $\lim_{i\rightarrow\infty}p_i=\infty$.

We consider the  rescaled sequence  $(M_{p_i},\lambda_{p_i}^2\,g_{\mathbb{S}^3},x_{p_i})$, where $x_{p_i}\in\Sigma_{p_i}$ and $\lambda_{p_i}>0$ is such that $\lim_{i\rightarrow
\infty}\lambda_{p_i} \text{inj}_{x_p}M_{p_i}>0$. Similarly, we consider $(\Sigma_{p_i},x_{p_i})\subset (M_{p_i},\lambda_{p_i}^2\,g_{\mathbb{S}^3},x_{p_i})$. By Cheeger-Gromov's compactness theorem,  there exists a subsequence $\{M_{p_i}\}_{i\in\mathbb{N}}$ which converges in the Cheeger-Gromov sense to a flat manifold $(M,\delta,x_{\infty})$.

\begin{lemma}\label{local area bounds 2}
\textit{	Let $(M_{p_i},\lambda_{p_i}^2\,g_{\mathbb{S}^3},x_{p_i})\xrightarrow{C-G}(M,\delta,x_{\infty})$ as above and assume that $\liminf_{i\rightarrow \infty}\lambda_{p_i}c(\Sigma_{p_i})>0$. Then $\{(\Sigma_{p_i},\lambda_{p_i}^2\,g_{\mathbb{S}^3},x_{p_i})\}_{i\in\mathbb{N}}$ satisfies local area bounds in $B_{R}(x_i)$ for some $R>0$.}
\end{lemma}
\begin{proof}
	As $\Sigma_p^{\prime}$ is $G_p$ invariant, then  $F:\Sigma_p^{\prime}\times [0,c(\Sigma_p^{\prime}))\rightarrow \mathbb{S}^3$ is also $G_p$ invariant. Hence, it makes sense to consider $F:\Sigma_p\times [0,c(\Sigma_p))\rightarrow M_p$ which is a diffeomorphism onto its image by Proposition \ref{Rolling theorem}. Let $r<\frac{1}{4}\min\{1,\liminf_{i\rightarrow \infty}\lambda_{p_i}c(\Sigma_{p_i})\}$, then for every $y_i\in \Sigma_{p_i}\cap B_R(x_i)$ we have that $\text{Vol}(B_{\frac{2r}{\lambda_{p_i}}}(y_p))\geq C_1\,r\, \text{Area}(\Sigma_p\cap B_\frac{r}{\lambda_{p_i}}(y_p))$ by Lemma \ref{local area bounds}. Since this formula is scale invariant,  the lemma is proved.
\end{proof}

\begin{lemma}\label{bounded curvature}
\textit{Let $A_p$ be the second fundamental form of $\Sigma_p$ in $M_p$. There exist $C>0$ such that $\sup_{\Sigma_p}|A_p|_{\lambda_p^2g_{\mathbb{S}^3}}=\sup_{\Sigma_p}\frac{1}{\lambda_p^2}|A_p|^2\leq C$.}
\end{lemma}
\begin{proof}
Let $y_p\in \Sigma_p \subset M_p$ be such that $|A_p|(y_p)=\max_{\Sigma_p}|A_p|^2$ and define the quantity 
$\rho_p= \max_{\Sigma_p}|A_p|(y_p)$. Arguing by contradiction, we assume that $\frac{\rho_p}{\lambda_p}\rightarrow \infty$. We consider the surface $\widehat{\Sigma}_p=(\Sigma_p,y_p) \subset (M_p),\rho_p^2\,g_{\mathbb{S}^3},y_p)$.
Under this scale, the sequence  $(M_p),\rho_p^2\,g_{\mathbb{S}^3},y_p)$  converges to $(\mathbb{R}^3,\delta,0)$ as $p \rightarrow \infty$. Moreover, the surface $\widehat{\Sigma}_p$ satisfies  $\max_{\widehat{\Sigma}_p}|A_p^{\prime}(x)|^2=|A_p^{\prime}(0)|^2=1$ and enjoys local area bounds by previous lemma.
By Proposition \ref{compactness}, $\widehat{\Sigma}_p$ converges to a non-flat properly embedded minimal surface $\Sigma_{\infty}\subset \mathbb{R}^3$ of index one. 
The convergence is with multiplicity one. Indeed, applying Proposition \ref{Rolling theorem} to $N$ and $-N$ we obtain that $F: (\Sigma_p,x_p)\times (-\alpha, \alpha)\rightarrow (M_{p_i}, \rho_{p_i}^2 g_{\mathbb{S}^3},x_p)$, with $0<\alpha < \liminf_{i\rightarrow \infty}\rho_{p_i} c(\Sigma_{p_i})$, is a diffeomorphism onto its image. Hence, there exists
a tubular neighbourhood of radius $\alpha$ around each $\Sigma_{p_i}$ in $(M_{p_i},\rho_{p_i}^2g_{\mathbb{S}^3})$ and  the convergence is with multiplicity one.  Since $g(\Sigma_{p_i})=3$, Lemma \ref{total curvature inequality} implies that $\int_{\Sigma_{\infty}}K_{\infty}\,d_{\Sigma_{\infty}}=0$. This contradicts  $|A_{\Sigma_{\infty}}|(0)=1$.
\end{proof}

Combining Lemma  \ref{local area bounds 2}, Lemma \ref{bounded curvature}, and Proposition \ref{compactness} we obtain:
\begin{lemma}\label{convergence}
	\textit{
		There exist  a properly embedded orientable minimal surface 
		$\Sigma_{\infty}\subset (M,\delta,x_{\infty})$} such that:
	\[\{\Sigma_{p_l}\}_{l \in \mathbb{N}}\subset (M_{p_l},\lambda_{p_l}^2\,g_{\mathbb{S}^3},x_{p_l}) \rightarrow 
	\Sigma_{\infty}\, \text{in the}\,\,\, C^k\,\,\, \text{topology}.
	\] 
	The convergence is with multiplicity one and the Morse index of $\Sigma_{\infty}$ is at most one.
\end{lemma}

\begin{lemma}\label{not totally geodesic}
	\textit{Let  $x_p \in \Sigma_p$ be  such that $\sup_{\Sigma_p}|A_p|=|A_p|(x_p)$. If $\lim_{p\rightarrow\infty}\lambda_p\,c(\Sigma_p)<\infty$, then
		\[\lim_{p \rightarrow\infty} \frac{|A_p|^2(x_p)}{\lambda_p^2}> 0.\] }
\end{lemma}
\begin{proof}
	As $\lim_{p\rightarrow\infty}\lambda_p\,c(\Sigma_p)<\infty$, there exists a positive constant $C$ such that $c(\Sigma_{p_i})\leq \frac{C\pi}{\lambda_{p_i}}$ for every $i\geq 1$. Hence,
	\begin{eqnarray*}
		c(\Sigma_p)\leq \frac{C\pi}{\lambda_p} \Leftrightarrow
		\arctan\bigg(\frac{1}{\lambda_2(x_p)}\bigg)\leq \frac{C\pi}{\lambda_p}\Leftrightarrow \lambda_2(x_p)\geq \frac{1}{\tan(\frac{C\pi}{\lambda_p})},
	\end{eqnarray*}
	where $\lambda_2(x)$ is the largest principal curvature of $\Sigma_p$ at $x$. Therefore,
	\[\lim_{p\rightarrow\infty} \frac{|A_p|^2(x_p^{\prime})}{\lambda_p^2}= \lim_{p\rightarrow\infty}\frac{2\lambda_2^2(x_p)}{\lambda_p^2}\geq \lim_{p\rightarrow \infty}\frac{2}{\lambda_p^2 \tan^2(\frac{C\pi}{\lambda_p})}=\frac{2C^2}{\pi^2}\]
	and the lemma is proved.
\end{proof}

\begin{lemma}\label{limit is a cylinder}
\textit{	If for each $p$ there exists  $\lambda_p$ such that   $\text{inj}_{x}M_p \geq \frac{C}{\lambda_p}$ for every  $x\in \Sigma_p$, then $\lim_{p\rightarrow\infty} \lambda_p\, c(\Sigma_p)=\infty$.}
\end{lemma}
\begin{proof}
	Let  $x_p \in \Sigma_p$  be such that $\sup_{\Sigma_p}|A_p|=|A_p|(x_p)$.
	By Lemma \ref{convergence}, $(M_p,\lambda_p^2g_0,x_p)\rightarrow (M,\delta,x_{\infty})$ in the Cheeger-Gromov convergence and   $\Sigma_p\rightarrow \Sigma_{\infty}$ in $(M,\delta,x_{\infty})$. If $\lim_{p\rightarrow\infty} \lambda_p\, c(\Sigma_p)<\infty$, then, by Lemma \ref{not totally geodesic}, $\Sigma_{\infty}$ is not totally geodesic. This contradicts Lemma \ref{total curvature inequality}.
\end{proof}

By Theorem \ref{classification of elliptic manifolds}, we may assume that  the subsequence $\{M_{p_i}\}_{i\in \mathbb{N}}$ satisfies either Case I, II, or III below:

\begin{flushleft} 
	Case I: The sequence $\{M_{p_i}\}_{i \in \mathbb{N}}$ is such that $H_2^p=\pi_2(\varphi(G_p))$ is  either $T$, $O$, or $I$.
\end{flushleft}

\begin{lemma}\label{injective radius homology sphere}
\textit{	If $M_p$ is such that $H_2^p=T$, $O$, or $I$, then $\text{inj}_xM_p=O(\frac{1}{p})$ for every $x\in M_p$. }
\end{lemma}
\begin{proof}
	Since the group $G_p$ preserves the Hopf fibers, the Hopf fibers have size  $O(\frac{1}{p})$.
	Let $h: (M_p,p^2\,g_{\mathbb{S}^3})\rightarrow (\mathbb{S}^2(\frac{1}{2})/H_2^p,p^2\,g_{\mathbb{S}^2})$ be the Hopf fibration.  Let $B_r(x_p)$ be the  ball of radius $r$ in $(M_p,p^2\,g_{\mathbb{S}^3})$. Since $H_2^p=T$, $O$, or $I$, there exists $c_0>0$ such that   $\text{vol}(h(B_r(x_p))\geq c_0 r^2$. By the co-area formula,
\begin{eqnarray*}
\text{vol}(B_{2r}(x_p)) \geq \int_{h(B_r(x_p))} \mathcal{H}^1(h^{-1}(y))\,d\mathcal{H}^2(y)\geq C\,c_0\,r^2.
\end{eqnarray*}
Cheeger's inequality implies that $\text{inj}_{x_p}(M_p,p^2\,g_{\mathbb{S}^3})\geq i_0$ for $i_0>0$.
\end{proof}

Since $g(\Sigma_p)\geq 3$, there exist a point $y_p\in\Sigma_p$ such that the Hopf fiber through $y_p$ is orthogonal to $\Sigma_p$. If we parametrize such fiber by $\gamma:[0,2\pi]\rightarrow M_p$, then the map $F$ from Proposition \ref{Rolling theorem} satisfies $F(y_p,t)=\gamma(t)$. It follows from Lemma \ref{injective radius homology sphere} that $c(\Sigma_p)\leq \frac{C}{p}$. This contradicts Lemma \ref{limit is a cylinder}.

\begin{flushleft} 
	Case II: The sequence $\{M_{p_i}\}_{i \in \mathbb{N}}$ is such that $H_2^p=\pi_2(\varphi(G_p))$ is  $\mathbb{Z}_m$.
\end{flushleft}
This corresponds to a subsequence of Lens spaces $L(p_i,q_i)$. The next lemma is useful for  the analysis of this case:
\begin{lemma}\label{lens space}
\textit{If $M_p=L(p,q)$ and $\text{diameter}(T_{\frac{\pi}{4}}/\mathbb{Z}_p)>\varepsilon$ for every $p$, then $\text{inj}_{x_p}M_p=O(\frac{1}{p})$ and $(M_p,p^2\,g_{\mathbb{S}^3},x_p)\xrightarrow{C-G}(\mathbb{S}^1\times \mathbb{R}^2,\delta, x_{\infty})$. Moreover, there exist a unit vector field $X\in \mathcal{X}(\mathbb{S}^3)$ which is $\mathbb{Z}_p$ invariant  and such that its orbits converge to the standard $\mathbb{S}^1$ fibers of $\mathbb{S}^1\times \mathbb{R}^2$.}
\end{lemma}
\begin{proof}
See Section 3 in \cite{V}.
\end{proof}

For  subsequences satisfying Lemma \ref{lens space}, we pick  $y_p\in \Sigma_p$ such that $g_{\mathbb{S}^3}(N(y_p),X(y_p))=\pm 1$. The existence of $y_p$ is from the Poincar\'{e}-Hopf Index Theorem applied to the vector field $X^{T}\in\mathcal{X}(\Sigma_p)$. Applying Lemmas \ref{convergence}  and \ref{total curvature inequality}, we conclude that $\Sigma_{\infty}$   is an union of  planes orthogonal to the  fibers of $\mathbb{S}^1\times \mathbb{R}^2$. This implies that  $\lim_{p\rightarrow\infty}p\,c(\Sigma_p)<\infty$ which contradicts Lemma \ref{limit is a cylinder}.

It remains to study subsequences of Lens spaces $L(p_i,q_i)$ such that \[\lim_{i\rightarrow\infty}\text{diameter}(T_{\frac{\pi}{4}}/\mathbb{Z}_{p_i})=0.\] 
Let us prove that  the pre-image of $\Sigma_p$ in $\mathbb{S}^3$, denoted by $\hat{\Sigma}_p$, converges in the Hausdorff sense to $T_{\frac{\pi}{4}}$ as $p\rightarrow \infty$. 
\begin{lemma}\label{injective radius uniform dense}
	\[\lim_{i \rightarrow \infty} d_{H}(\hat{\Sigma}_{p_i}, T_{\frac{\pi}{4}})= 0.\]
\end{lemma}
\begin{proof}
 Without loss of generality, we assume that $\Sigma_p \cap A_p$ is stable, where $A_p=\{x\in L(p,q): r(x)\geq \frac{\pi}{4} \}$ and  $r(x)=r$ if, and only if, $x\in T_r$. Let us define the quantities $a=\liminf_{p\rightarrow \infty}\inf \{r(x): x\in \Sigma_p\}$ and $b=\limsup_{p\rightarrow \infty}\sup \{r(x): x\in \Sigma_p\}$. If $b< \frac{\pi}{2}$, then $T_{b}$ can be obtained as a limit of $\hat{\Sigma}_p$ as $p\rightarrow \infty$ since the curvature of $\Sigma_p\cap \hat{A}_p$ is uniformly bounded and since each orbit of $\mathbb{Z}_{p_i}$ is becoming dense on the Clifford torus that contains it. Thus,  $b=\frac{\pi}{4}$ which implies that $a=\frac{\pi}{4}$ and the lemma is proved in this case. Indeed, if $a<\frac{\pi}{4}$, then $T_{\frac{\pi}{4}}$ would be a stable minimal surface, contradiction.
 Hence, we may assume that $b=\frac{\pi}{2}$. 
 
 First we study the case where  $\Sigma_p\cap T_{\frac{\pi}{2}}= \emptyset$ for every $p$. Let $x_p\in \hat{\Sigma}_p$  be  the closest point to $T_{\frac{\pi}{2}}$. By the stability assumption, the connected components of $\hat{\Sigma}_p$ in $\hat{A}_p$  converge to  leafs of  a minimal lamination $F$ in $\hat{A}_p$. Since $T_{\frac{\pi}{2}}$ is tangent to every such leaf that it intersects, we conclude that $T_{\frac{\pi}{2}}$ is contained in a leaf $F_{\alpha}$. Let $\Gamma_p\subset \mathbb{S}^3$ be a minimal torus  containing the geodesics $T_0$ and $T_{\frac{\pi}{2}}$ and  perpendicular to $\hat{\Sigma}_p$ at $x_p$. 
 The minimal tori $\Gamma_p$  is a leaf of the singular lamination $E=\{E_{\beta}\}$ by the union of minimal tori  containing  $T_0$ and $T_{\frac{\pi}{2}}$. 
 By compactness, $\Gamma_p$ converge to a leaf $E_{\beta}$ perpendicular to  $F_{\alpha}$ along $T_{\frac{\pi}{2}}$. Consequently, there exist another leaf $E_{\beta_1}$ which is tangent to $F_{\alpha}$ along $T_{\frac{\pi}{2}}$. By the analytical continuation, the lamination $F$ coincide with the singular lamination $E$, contradiction. 
 
 Now we study  the case $\Sigma_p\cap T_{\frac{\pi}{2}}\neq \emptyset$. By choosing $x_p \in \Sigma_p\cap T_{\frac{\pi}{2}}$, we have that $(L(p,q),p^2g_0,x_p)\rightarrow(M,\delta,x_{\infty})$, where $M$ is a quotient of $\mathbb{R}^3$ by a screw motion with angle $\theta$ and $(\Sigma_p,p^2g_0,x_p) \rightarrow (\Sigma_{\infty},\delta,x_{\infty})$, where $\Sigma_{\infty}\subset M$ is totally geodesic by Lemma \ref{total curvature inequality}. If $\Sigma_{\infty}$ is a plane, then $\lim_{p \rightarrow \infty}p\,c(\Sigma_p)<\infty$. As this contradicts Lemma \ref{limit is a cylinder} (note that $\text{inj}_xL(p,q)\geq \frac{\pi}{p}$ for every $x$), we conclude that  $\Sigma_{\infty}$ is flat cylinder.
 It is enough to proving that  $\theta\neq 0$,  since there are no totally geodesic cylinders in $M$ in this case. Let $T_{r_p}$ be the Clifford torus such that $\lim_{p\rightarrow \infty}p\, d_{L(p,q)}(T_{r_p}, T_{\frac{\pi}{2}})= c_0$. It follows that $(T_{r_p},p^2\,g_0)$ converges to a tube of radius $c_0$ around the central fiber  in $(M,\delta)$  through $x_{\infty}$. Recall the Hopf fibration  $h: \mathbb{S}^3\rightarrow \mathbb{S}^2(\frac{1}{2})$. If $\gamma_p: [0,1]\rightarrow \mathbb{S}^3$ is the geodesic segment such that  $|\gamma_p|=2\,\text{inj}_{\gamma_p(0)}L(p,q)$ with $\gamma_p(0)\in T_{r_p}$, then $h(\gamma_p)$ is a geodesic in $\mathbb{S}^2(\frac{1}{2})$ whose extremities determine  an arc $\beta_p:[0,1]\rightarrow h(T_{r_p})$. 
 Under the scale $\lambda=p^2$ of the round metric $g_0$, $\beta_p$ converge to an arc $\beta$ in the geodesic circle  of radius $c_0$   centered at the origin in $\mathbb{R}^2$ and $h(\gamma_p)$ converges to a linear segment $\gamma$  whose extremities are those of $\beta$. The angle $\frac{|\beta|}{c_0}$ is independent of the choice of $c_0$. Hence, $|\gamma|$ increases as $c_0$ increases. In particular, the injective radius of $M$ is not constant and, hence, $\theta \neq 0$.  
\end{proof}
 By Lemma \ref{injective radius uniform dense},  there exists $\lambda_p$ such that  $\frac{1}{C} \frac{1}{\lambda_p}\leq \text{inj}_{x}L(p,q)= C\frac{1}{\lambda_p}$ for every $x \in \Sigma_p$. The constant $C>0$ is independent of $p$ and $q$.
As before,  $(L(p,q),\lambda_p^2g_0,x_p)\rightarrow(M,\delta,x_{\infty})$ and  $\Sigma_p\rightarrow \Sigma_{\infty}$, where $\Sigma_{\infty}$ is a totally geodesic surface in $(M,\delta)$  by Lemma \ref{total curvature inequality}. If $M$ is diffeomorphic to  $T^2\times \mathbb{R}$, then $\lim_{p \rightarrow\infty} \lambda_p\,c(\Sigma_p)<\infty$ and we obtain a contradiction   with Lemma \ref{limit is a cylinder}. Similar argument for the case  when $M=\mathbb{R}^3/T_v$ and $\Sigma_{\infty}$ an union of planes. Therefore, we assume, regardless the choices of base points, that $M$ is diffeomorphic to $\mathbb{S}^1\times \mathbb{R}^2$ and  that $\Sigma_{\infty}$ is a totally geodesic $\mathbb{S}^1\times \mathbb{R}$. Let us show that this is incompatible with  the assumption that $\text{genus}(\Sigma_p)>1$. 
\begin{lemma}
	\textit{For each $j$, let  $\Sigma_{j}$ be a closed minimal surface of genus $g$ in $M_j$  and assume that $(M_j,\lambda_j^2g_0,x_j) \rightarrow (\mathbb{S}^1\times \mathbb{R}^2,\delta,x_{\infty})$ and that $\Sigma_j \rightarrow \mathbb{S}^1\times \mathbb{R}$ for every choice of  base points $x_j\in \Sigma_j$. Then $g=1$.}
\end{lemma}
\begin{proof}
It follows from the assumptions, that there exist positive constants $C_1$ and $C_2$ such that $\frac{C_1}{\lambda_j}\leq \text{inj}_x\Sigma_{j}\leq \frac{C_2}{\lambda_j}$ and $\frac{C_1}{\lambda_j}\leq \text{inj}_xM_{j}\leq \frac{C_2}{\lambda_j}$ for every $x \in \Sigma_j$ and every $j$. For each $j$, let $\mathcal{F}_j=\{B_1,\ldots, B_{N_j}\}$ be a maximal disjoint collection of balls  $B_i=B_{\frac{R}{\lambda_j}}(x_{ij})$ in $M_j$ where $x_{ij}\in \Sigma_j$ and $R> 4\,C_2$. By the assumption of the lemma, there exists $j_0$ such that $\Sigma_j\cap B_{\frac{R}{\lambda_j}}(x_{ij})$ is an annular surface for every $j\geq j_0$. For $j$ sufficiently large, let $K_j$ be a connected component of $\Sigma_j- \cup_{i=1}^{N_j}B_i$ and take $y_j \in K_j$. By assumption,   $(\Sigma_j,\lambda_j^2g_0,y_j)\rightarrow (\mathbb{S}^1\times \mathbb{R},\delta, y_{\infty})$. We consider $\mathcal{F}_{\infty}$  the disjoint collection of regions in $\Sigma_{\infty}$ obtained as the limit of $\Sigma_j\cap B_{ij}$. Note that each element of $F_{\infty}$ is the intersection of geodesic balls in $\mathbb{S}^1\times \mathbb{R}^2$ centered on $\Sigma_{\infty}$ and radius $R\in [C_1,C_2]$, hence, an annulus where each boundary component generates $\pi_1(\Sigma_{\infty})$. Moreover, each connected component of $\Sigma_{\infty}-\mathcal{F}_{\infty}$ is compact by the maximality of $\mathcal{F}_j$. Since $K_j$ is connected and $y_{\infty}\notin \mathcal{F}_{\infty}$,  we conclude that $K_{\infty}$ is also an annulus. Hence, there exists an integer $j_2$ such that $\Sigma_j$ is an union of disjoint annulus for every $j\geq j_2$. By the Gauss-Bonnet Theorem, $\text{genus}(\Sigma_j)=1$.
\end{proof}

\begin{flushleft} 
	Case III: The sequence $\{M_{p_i}\}_{i \in \mathbb{N}}$ is such that $H_2^p=\pi_2(\varphi(G_p))$ is  $\mathbb{D}_{2n}$.
\end{flushleft}

The spherical space forms in this case are  double covered by  lens spaces. The arguments  in Case II  apply  \textit{mutatis mutandis}. 
\end{proof}

\end{document}